\documentclass[letterpaper, 10 pt, conference]{ieeeconf}  

\usepackage{amsmath,amssymb,amsfonts}
\usepackage{algorithm}
\usepackage{graphicx}
\usepackage{textcomp}
\usepackage{url}
\usepackage{color}
\usepackage{bm}
\usepackage{cases}
\usepackage{lipsum}
\usepackage{epstopdf}

\usepackage{amsopn}
\DeclareMathOperator{\diag}{diag}

\usepackage{amsthm}

\usepackage{enumerate}
\usepackage{cancel}
\usepackage{mathrsfs}
\usepackage{mathdots}
\usepackage{euscript}
\usepackage{amscd}
\usepackage{placeins}
\usepackage[T1]{fontenc}
\usepackage{booktabs}

\usepackage{lineno}
\modulolinenumbers[5]

\usepackage[font=small,labelfont=bf]{caption}
\usepackage{subcaption}

\usepackage{algpseudocode}

\usepackage{multirow}
\usepackage{tikz}
\usetikzlibrary{arrows}
\usepackage{verbatim}

\newtheorem{theorem}{Theorem}

\newtheorem{proposition}{Proposition}

\theoremstyle{definition}

\theoremstyle{remark}
\newtheorem{remark}{Remark}

\newcommand{\bmat}{\left[ \begin{matrix}}
	\newcommand{\emat}{\end{matrix} \right]}

\DeclareMathOperator{\trace}{tr}

\newcommand{\Rbb}{\mathbb R}

\newcommand{\xb}{\mathbf  x}
\newcommand{\yb}{\mathbf  y}

\newcommand{\wb}{\mathbf  w}

\newcommand{\eb}{\mathbf  e}

\newcommand{\ub}{\mathbf  u}

\newcommand{\zerob}{\mathbf 0}

\newcommand{\Ib}{\mathbf I}

\newcommand{\Lb}{\mathbf L}

\newcommand{\Pb}{\mathbf P}

\newcommand{\Sb}{\mathbf S}

\newcommand{\Ub}{\mathbf U}

\newcommand{\Yb}{\mathbf Y}
\newcommand{\Zb}{\mathbf Z}
\newcommand{\Xb}{\mathbf  X}
\newcommand{\Wb}{\mathbf W}

\newcommand{\mub}{\boldsymbol{\mu}}

\newcommand{\Gammab}{\boldsymbol{\Gamma}}
\newcommand{\Sigmab}{\boldsymbol{\Sigma}}
\newcommand{\Deltab}{\boldsymbol{\Delta}}
\newcommand{\Lambdab}{\boldsymbol{\Lambda}}

\newcommand{\sign}[1]{\mathrm{sign}(#1)}

\newcommand{\Qcal}{\mathcal{Q}}
\newcommand{\Acal}{\mathcal{A}}

\newcommand{\Zcal}{\mathcal{Z}}
\newcommand{\Dcal}{\mathcal{D}}

\newcommand{\KL}{\mathrm{KL}}
\newcommand{\F}{\mathrm{F}}

\newcommand{\vrm}{\mathrm{v}}
\newcommand{\trm}{\mathrm{t}}

\newcommand{\opt}{\text{opt}}

\IEEEoverridecommandlockouts                              
\title{\LARGE \bf
$\ell_0$ FACTOR ANALYSIS}

\author{Linyang Wang, Wanquan Liu, and Bin Zhu
\thanks{This work was supported in part by Shenzhen Science and Technology Program (Grant No.~202206193000001-20220817184157001), the Fundamental Research Funds for the Central Universities, and the ‘‘Hundred-Talent Program’’ of Sun Yat-sen University.}
\thanks{The authors are with the School of Intelligent Systems Engineering, Sun Yat-sen University, Gongchang Road 66, 518107 Shenzhen, China. 
Emails: {\tt\small wangly227@mail2.sysu.edu.cn} (L. Wang),
{\tt\small \{liuwq63, zhub26\}@mail.sysu.edu.cn} (W. Liu and B. Zhu).
}
}

\begin{document}

\maketitle
\thispagestyle{empty}
\pagestyle{empty}

\begin{abstract}

Factor Analysis is about finding a low-rank plus sparse additive decomposition from a noisy estimate of the signal covariance matrix.
In order to get such a decomposition, we formulate an optimization problem using the nuclear norm for the low-rank component, the $\ell_0$ norm for the sparse component, and the Kullback--Leibler divergence to control the residual in the sample covariance matrix. 
An alternating minimization algorithm is designed for the solution of the optimization problem.
The effectiveness of the algorithm is verified via simulations on synthetic and real datasets.
\end{abstract}

\section{Introduction}


Factor Analysis (FA) is a classic topic in Psychology, Econometrics, Signal Processing, Machine Learning, and Control, see e.g., \cite{bai2002determining,lam2012factor,bottegal2014modeling,bertsimas2017certifiably,zorzi2017sparse,ciccone2018factor} and the references therein. More specifically, it concerns the following observational model:
\begin{equation}\label{generate_model}
	\yb_i = \mub + \Gammab \ub_i + \wb_i,\quad i=1,2,\dots,N,
\end{equation}
where $\yb_i\in\Rbb^p$ is an observed vector, $\mub\in\Rbb^p$ is a mean vector, $\Gammab\in\Rbb^{p\times r}$ is a ``factor loading'' matrix having linearly independently columns, the random vector $\ub_i\sim N(0,\Ib_r)$ stands for the hidden factors, and $\wb_i\sim N(0,\Sb^*)$ is the additive noise (with an \emph{unknown} covariance matrix $\Sb^*$) independent of $\ub_i$. It is a widely used form of linear dimensionality reduction because typically one has $r\ll p$. Given $N$ i.i.d.~samples $\yb_i$ from the model, the problem is to estimate the loading matrix $\Gammab$, or equivalently, the rank $r$ matrix $\Lb^*=\Gammab\Gammab^\top\in\Rbb^{p\times p}$ and the noise covariance matrix $\Sb^*$.

Assume for simplicity that the mean vector $\mub=\zerob$. It is easy to calculate the covariance matrix of $\yb_i$ as 
\begin{equation}\label{Sigma_decomp}
	\Sigmab = \Gammab\Gammab^\top+\Sb^* = \Lb^* + \Sb^*.
\end{equation}
In the special case where $\Sb^*=\sigma^2 \Ib_p$, the problem reduces to the standard \emph{Principal Component Analysis} (PCA). A typical assumption in FA is $\Sb^*$ being diagonal. Then \eqref{Sigma_decomp} becomes a kind of ``low-rank plus sparse'' matrix decomposition. In practice, of course the covariance matrix $\Sigmab$ must be replaced by its estimate from samples, say
\begin{equation}\label{sample_cov}
	\hat{\Sigmab} = \frac{1}{N} \sum_{i=1}^{N} \yb_i \yb_i^\top = \Lb^* + \Sb^* + \Wb,
\end{equation}
where $\Wb$ is a residual matrix. \emph{Throughout this paper, we assume that the estimate $\hat{\Sigmab}$ is positive definite.}

The recent paper \cite{ciccone2018factor} casts FA as a constrained convex optimization problem.  Their idea is to find a covariance matrix $\Sigmab^*$ in the ``neighborhood'' of $\hat{\Sigmab}$, as described by the \emph{Kullback--Leibler (KL) divergence}, such that the decomposition \eqref{Sigma_decomp} has (possibly) a minimum-rank $\Lb^*$ and a diagonal $\Sb^*$.
In addition, the \emph{nuclear norm} is used as a relaxation of the rank function in order to make the problem tractable.
Inspired by their work, in the current paper we shall relax the constraint that the component $\Sb$ is diagonal, and instead require $\Sb\succ 0$ to be sparse as measured by the most natural $\ell_0$ norm. 
In order to solve the resulting nonconvex nonsmooth optimization problem, we propose an alternating minimization scheme for the iterative updates of $\Lb$ and $\Sb$. Simulations on synthetic and real data show that our algorithm is effective and robust in finding the number of hidden factors, i.e., the rank of $\Lb^*$.


\subsection{Some related works}

The recovery of the matrix pair $\Lb^*, \Sb^*$ from the noisy estimate $\hat{\Sigmab}$ is reminiscent of the extensive research on robust PCA, see \cite{candes2011robust,chandrasekaran2011rank,hsu2011robust,agarwal2012noisy,wen2019robust,chen2021bridging}. 
However, these works mostly use the $\ell_1$ norm as a convex surrogate of the $\ell_0$ norm to enforce sparsity. Moreover, they do not pay special attention to covariance matrices, and in general, rectangular matrices are allowed.

Another related work is \cite{marjanovic2015l0} which uses the $\ell_0$ norm for inverse covariance estimation.
	In the literature, there have also been research based on proximity operators to ensure the sparse property of the optimization variable,
	such as $\ell_q$ thresholding \cite{marjanovic2012l_q}, hard thresholding \cite{ulfarsson2015sparse}, and $q$-shrinkage \cite{woodworth2016compressed}. In this paper, however, the method utilized  to handle the $\ell_0$ norm is distinct from the techniques involving proximity operators. 

\subsection{Notation}
Bold uppercase letters like $\Xb$ represent matrices, and bold lowercase letters like $\xb$ are reserved for vectors.
Given a square matrix $\mathbf{X}_{p\times p}=[x_{ij}]$, 
$\Vert \Xb\Vert_{\F}$ denotes the Frobenius norm. 
The $i$-th column of $\Xb$ is $\mathbf{x}_{[i]}$. 
We write $\mathbf{X}\succeq 0 $ and $\mathbf{X}\succ 0 $ to indicate that $\mathbf{X}$ is positive semidefinite and positive definite, respectively. 
The indicator function $\mathbb{I} (\cdot)$ returns $1$ if the statement in the parenthesis is logically true and $0$ otherwise. 
$\mathbf{e}_i$ is a unit column vector with the $i$-th entry as $1$ and all other entries as $0$. 
$\mathbf{U}_{ij}$ denotes a matrix with two unit columns $\bmat\mathbf{e}_i & \mathbf{e}_j\emat$.


\section{Problem Formulation}\label{sec:prob}

In this work, we consider the following optimization problem for the additive matrix decomposition in accordance with \eqref{sample_cov}:
\begin{subequations}\label{opt_formulation}
	\begin{align}
		& \underset{\Sigmab, \mathbf{L}, \mathbf{S}}{\operatorname{min}} & & \operatorname{tr} (\mathbf{L})+\lambda\|\mathbf{S}\|_0 \\
		& \ \ \text{s.t.} & &\mathbf{L}\succeq 0, \mathbf{S} \succ 0 \label{constraint_positive} \\
		& & & \mathbf{\Sigma}=\mathbf{L}+\mathbf{S} \label{constraint_additive}\\
		& & & \mathcal{D}_{\mathrm{KL}} (\mathbf{\Sigma} \| \hat{\mathbf{\Sigma}})\leq \delta,
	\end{align}
\end{subequations}
where, all matrices are $p$ by $p$, $\lambda$ and $\delta$ are positive parameters,
\begin{itemize}
	\item $\trace(\Lb)=\|\Lb\|_\star$ is the nuclear norm for positive semidefinite matrices, 
	\item $\|\mathbf{S}\|_0=\sum_{i=1}^{p} \sum_{j=1}^{p} \mathbb{I}(s_{ij}\neq0)$ is the elementwise $\ell _0$-norm which counts the number of nonzero entries in $\mathbf{S}$,
	\item $\hat{\Sigmab}\succ 0$ is the sample covariance matrix in \eqref{sample_cov},
	\item 
		$\Dcal_{\KL}(\Sigmab||\hat{\Sigmab}) := 
		\log\det(\Sigmab^{-1} \hat{\Sigmab}) + \trace(\Sigmab\hat{\Sigmab}^{-1}) -p$
	is the KL divergence between two positive definite matrices.
\end{itemize}


The idea is to find a covariance matrix $\mathbf{\Sigma}$, close to its estimate $\hat{\mathbf{\Sigma}}$ as measured by the KL divergence, achieving the combined objective of having a low-rank $\mathbf{L}^*,$ and a sparse $\mathbf{S}^* $ in the additive decomposition \eqref{Sigma_decomp}. 
One can alternatively deal with a Lagrangian-type formulation
\begin{equation}\label{lagrange-type}
	\underset{\Sigmab, \mathbf{L}, \mathbf{S}}{\operatorname{min}} \ \operatorname{tr}(\mathbf{L})+\lambda\|\mathbf{S}\|_0+
	\mu \mathcal{D}_{\mathrm{KL}}(\mathbf{\Sigma}\| \hat{\mathbf{\Sigma}}) \\
	\ \text { s.t. }\ \eqref{constraint_positive}\ \text{ and }\ \eqref{constraint_additive}
\end{equation}
where $\mu >0 $ is another (regularization) parameter.
The above problem \eqref{lagrange-type} will be the focus of investigation in the remaining part of the paper.
It is more convenient to eliminate the variable $\mathbf{\Sigma}$ by the sum of $\mathbf{L}$ and $\mathbf{S}$, yielding the equivalent form
\begin{equation}\label{Object}
	\underset{\mathbf{L}\succeq 0, \mathbf{S}\succ 0}{\min}\ 
	H(\mathbf{L},\mathbf{S}) := f(\mathbf{L},\mathbf{S})+\lambda\Vert \mathbf{S}\Vert _0
\end{equation}
where 
\begin{equation}
	f(\mathbf{L},\mathbf{S}) := \trace(\mathbf{L})+\mu\left[ \operatorname{tr}(\mathbf{L}+\mathbf{S}) \hat{\mathbf{\Sigma}}^{-1}- \log \operatorname{det}(\mathbf{L}+\mathbf{S})\right].
\end{equation}
	Notice that the non-convexity and non-differentiability of the objective function $H (\mathbf{L},\mathbf{S})$ is solely because of the term $\Vert \mathbf{S}\Vert _0$.
	In fact, the function $f(\mathbf{L},\mathbf{S})$ is smooth and convex separately in $\Lb$ or $\Sb$ when the other variable is held fixed.  
	Moreover, for a fixed  $\mathbf{S}$, $f(\Lb, \Sb)$ is strictly convex in $\Lb$.

\begin{remark}
	In comparison with the problem formulation in \cite{ciccone2018factor}, we do not impose the diagonal constraint on the sparse structure of $\Sb$ in \eqref{opt_formulation}. In this respect, our formulation is expected to be more flexible and realistic.
\end{remark}

	
\begin{theorem}\label{thm_solution_compare}
		Suppose that the $\ell_0$ regularization term $\lambda\|\Sb\|_0$ in \eqref{Object} is replaced by the $\ell_1$ regularization term $\tau\|\Sb\|_1$, and that $\Qcal_{\ell_1}(\tau)$ is the set of global minimizers of the $\ell_1$ relaxation of the problem \eqref{Object} with a parameter $\tau > 0$. Let $\Qcal_{\ell_0}(\lambda)$ be the set of all local minimizers of \eqref{Object} with a parameter $\lambda$. Then for any $(\hat \Lb, \hat \Sb)\in \Qcal_{\ell_1}(\tau)$, we have $ (\hat \Lb, \hat \Sb) \notin  \Qcal_{\ell_0} (\lambda)$ for any $\lambda>0$.
\end{theorem}

In plain words, Theorem \ref{thm_solution_compare} says that any solution to the $\ell_1$ relaxation of the problem \eqref{Object} will not be a local minimizer of the original problem \eqref{Object} no matter what the regularization parameter $\lambda$ is chosen.


%
%
%
%
%

\section{Algorithm Development}\label{sec:alg}



We employ an alternating minimization scheme to solve the optimization problem \eqref{Object}. More precisely, we update the current iterates $(\Lb^k, \Sb^k)$ as follows:
\begin{subequations}\label{L-S_update}
	\begin{align}
		\Sb^{k+1} & = \arg\min_{\Sb\succ 0}\ H(\Lb^k, \Sb), \label{S_update}\\
		\Lb^{k+1} & = \arg\min_{\Lb\succeq 0}\ H(\Lb, \Sb^{k+1})
		= \arg\min_{\Lb\succeq 0}\ f(\Lb, \Sb^{k+1}), \label{L_update}
	\end{align}
\end{subequations}
where the last equality is due to the fact that the second term in \eqref{Object} depends only on $\Sb$. The next subsections deal with the two subproblems.


\subsection{Updating $\mathbf{L}^{k+1}$}

For the subproblem \eqref{L_update}, we have the next result.

\begin{proposition}
	When $\Sb$ is held fixed, the objective function $H(\mathbf{L}, \mathbf{S})$ or $f(\mathbf{L}, \mathbf{S})$ is smooth and strictly convex in
	$\mathbf{L}$.
\end{proposition}


Therefore, the optimization problem \eqref{L_update} for $\mathbf{L}$ can be treated numerically using standard solvers for convex optimization. Here we have used CVX, a package for specifying and solving convex programs \cite{cvx,gb08}.

\subsection{Updating $\mathbf{S}^{k+1}$}

To handle the subproblem \eqref{S_update}, we propose a Coordinate Descent (CD) algorithm inspired by \cite{marjanovic2015l0}.
CD algorithm minimizes one selected entry with all others fixed in each iteration. After a complete round of CD updates for all the entries of $\Sb$, the counter $k$ will be increased by one.
The specific update equation for $\mathbf{S}^k=[s_{ij}^k]$ is given as follows:
\begin{equation}
	\Zb_{ij}\left(s_{i j}^{k+1}\right)=\boldsymbol{\mathbf{S}}^k +
	\left\{\begin{array}{cc}
		\delta\left(s_{i i}^{k+1}\right) \boldsymbol{\mathbf{e}}_i \boldsymbol{\mathbf{e}}_i^\top & \text { if } i=j \\
		\delta\left(s_{i j}^{k+1}\right) \boldsymbol{\mathbf{U}}_{i j} \boldsymbol{\mathbf{U}}_{j i}^\top & \text { otherwise },
	\end{array}\right.
\end{equation}
where $k$ denotes the number of complete rounds for CD, $\mathbf{Z}_{ij}(s_{i j}^{k+1})$ is the matrix with $s_{ij}$ updated in the $k$-th round, and $ \delta\left(s_{i j}^{k+1}\right)=s_{i j}^{k+1}-s_{i j}^k $ denotes the difference.

Next define $\mathbf{Y}^{k}:=(\mathbf{L}^{k}+\mathbf{S}^{k})^{-1}$ and
\begin{equation}
	\begin{aligned}
		{\phi_{ij}\left(s\right)} := 
		& -\mu \log \operatorname{det}\left(\mathbf{L}^k+\mathbf{Z}_{ij}\left(s\right)\right)+\mu s d_{i j} \\
		& + [ \mu s d_{i j} +2 \lambda \cdot \mathbb{I} \left(s \neq 0\right) ]\cdot \mathbb{I} \left(i \neq j \right)
	\end{aligned}
\end{equation}
for any $i,j$, where $d_{ij}$ denotes the {$(i, j)$} element of $\hat{\mathbf{\Sigma}}$, and $\phi_{ij}\left(s\right)$ represents the equivalent function to minimize for $s_{ij}^{k+1}$ while the other elements of $\Sb$ are held fixed. In other words, we have
\begin{equation}\label{argmin_s_ij}
	s_{ij}^{k+1} = \arg\min_{s} H(\mathbf{L}^k,\mathbf{Z}_{ij}(s)) =    \arg\min_{s} \phi_{ij}(s).
\end{equation}



\subsubsection{Minimization of $\phi_{ij}$ when $i = j$}

When one works on the diagonal elements of $\Sb$, the minimization problem in \eqref{argmin_s_ij} reduces to
\begin{equation}
	\arg \min _s \phi_{ii}\left(s\right) 
	= \arg \min _s -\log \operatorname{det}\left(\mathbf{L}^k+\mathbf{Z}_{ii}(s)\right)+d_{i i} s.
	\nonumber
\end{equation}
Clearly in the case of $i=j$, $\phi_{ii}(\cdot)$ is differentiable,
so the minimizers can be obtained by differentiating $\phi(\cdot)$:
\begin{equation}\label{deff-phi}
	\phi'(s)= -[(\mathbf{L}^k+\mathbf{Z}_{ii}(s))^{-1}]_{ii}+d_{i i}  =0.
\end{equation}
With the Sherman-Morrison-Woodbury formula \cite{henderson1981deriving}, we have
\begin{equation}\label{inv_ii}
		\begin{aligned}
		\left(\mathbf{L}^k+\mathbf{S}^k+\delta \mathbf{e}_i \mathbf{e}_i^T\right)^{-1}
				& =\mathbf{Y}^k-\frac{\delta \mathbf{Y}^k \mathbf{e}_i \mathbf{e}_i^\top \mathbf{Y}^k}{1+\delta y_{i i}^k} \\
				& 
		= \mathbf{Y}^k - \frac{\delta \mathbf{y}_{[i]}^k (\mathbf{y}_{[i]}^k)^\top}{(1+\delta y_{i i}^k)}.
			\end{aligned}
\end{equation}
Recall that $ \delta(s)=s -s_{i i}^k $, and then we have
\begin{equation}\label{ii-1}
	\left[(\mathbf{L}^{k}+\mathbf{Z}_{ii}(s))^{-1}\right]_{i i} = \frac{y_{i i}^k}{1+\delta\left(s\right) y_{i i}^k}.
\end{equation}
Substituting \eqref{ii-1} into \eqref{deff-phi} to solve for $s_{ii}^{k+1}$, the minimizer is given by
\begin{equation}\label{sii1}
	m_{i i}=s_{i i}^k+\frac{{y_{i i}^k-d_{i i}}}{y_{i i}^k d_{ii}} .
\end{equation}
Furthermore, we can verify that the matrix $\mathbf{L}^{k}+\mathbf{Z}_{ii}(m_{ii})$ has a positive determinant:
\begin{equation}\label{det_updated_ii}
	\det \left(\mathbf{L}^{k}+\mathbf{Z}(m_{ii})\right)=\operatorname{det}\left(\mathbf{L}^k+\mathbf{S}^k\right)\left( 1+\delta\left(m_{i i}\right) y_{i i}^k\right)>0,
\end{equation}
which is a necessary condition for being positive definite.

\subsubsection{Minimization of $\phi_{ij}$ when $i \neq j$}

For nondiagonal elements of $\Sb$, the problem \eqref{argmin_s_ij} becomes
\begin{equation}
		\arg \min _s \mu[-\log \operatorname{det}\left(\mathbf{L}^k+\mathbf{Z}_{ij}(s)\right)+2d_{i j}s]+2 \lambda \cdot \mathbb{I} (s \neq 0).
	\nonumber
\end{equation}
If $0$ is in the domain of the function $\phi_{ij}(\cdot)$, i.e., $\operatorname{det}(\mathbf{L}^k+\mathbf{Z}_{ij}(0))>0$, then $\phi_{i j}(\cdot)$ has one discontinuous point at $s=0$. Otherwise $\phi_{i j}(\cdot)$ will be smooth everywhere. 
Define the smooth part of $\phi_{i j}(\cdot)$ as 
\begin{equation}
	{g_{ij}(s)} :=\mu[-\log \operatorname{det}\left(\mathbf{L}^{k}+\mathbf{Z}_{ij}(s)\right)+2d_{i j}s]+2 \lambda. 
\end{equation}

Firstly, if $\det(\mathbf{L}^k+\mathbf{Z}_{i j}(0))>0$, the equivalent expression of  $\phi_{i j}(\cdot)$ is 
	$g_{i j}(s)\cdot \mathbb{I} \left(s \neq 0\right) +\left(g_{i j}(s)-2 \lambda\right) \cdot \mathbb{I} \left(s = 0\right)$. 
Obviously the a minimizer of $\phi_{i j}(\cdot)$ is either a minimizer of $g_{i j}(\cdot)$ or $s=0$. Since $g_{i j}(\cdot)$ is strictly convex and differentiable, we take the derivative of $g_{i j}(\cdot)$ to get a unique minimizer and then compare the function value to $\phi_{i j}(0)$. The stationary-point equation is given by
\begin{equation}\label{deff-phi2}
	g_{i j}'(s)= 2\mu[-{[(\mathbf{L}^k+\mathbf{Z}_{i j}(s))^{-1}]_{ij}}+d_{i j} ] =0	.
\end{equation}
Again by the Woodbury formula, we can obtain
\begin{equation}\label{inv_ij}
	\begin{aligned}
		& \left(\mathbf{L}^k+\mathbf{S}^k+\delta \mathbf{U}_{i j} \mathbf{U}_{j i}^\top\right)^{-1} \\
		= &\, \mathbf{Y}^k-\frac{\delta\left[\begin{array}{ll}
				\mathbf{y}_{[i]}^k & \mathbf{y}_{[j]}^k
				\end{array}\right]\left[\begin{array}{cc}
				1+\delta y_{i j}^k & -\delta y_{j j}^k \\
				-\delta y_{i i}^k & 1+\delta y_{i j}^k
				\end{array}\right]\left[\begin{array}{l}
				{\mathbf{y}_{[j]}^k}^\top \\
				{\mathbf{y}_{[i]}^k}^\top
				\end{array}\right]}{-\Delta_{i j}^k \delta^2+2 y_{i j} \delta+1},
		\end{aligned}
\end{equation}
Hence
\begin{equation}\label{ij-1}
	\left[(\mathbf{L}^{k}+\mathbf{Z}_{ij}(s))^{-1}\right]_{i j} = \frac{-\Delta_{i j}^k \delta\left(s\right)+y_{i j}^k}{-\Delta_{i j}^k \delta\left(s\right)^2+2 y_{i j}^k \delta\left(s\right)+1},
\end{equation}
where
\begin{equation}
	\Delta_{i j}^k:=\Delta_{i j}(\mathbf{Y}^k)=y_{i i}^{k} y_{j j}^{k}-{y_{i j}^{k}}^2>0
\end{equation}
is the determinant of a $2\times 2$ submatrix of $\Yb^k$.

\emph{Case (i).} When $d_{ij}= 0$, substituting \eqref{ij-1} into \eqref{deff-phi2}, the minimizer is given by
\begin{equation}\label{m_ij_case(i)}
	m_{i j}=s_{i j}^k+\frac{y_{i j}^k}{\Delta_{i j}^k } .
\end{equation}
Similar to \eqref{det_updated_ii}, the following calculation shows that the updated determinant is positive: $\det\left(\mathbf{L}^{k}  +\mathbf{Z}_{ij}(m_{ij})\right) =$
\begin{equation}
	\det(\mathbf{L}^k+\mathbf{S}^k)\left(-\Delta_{ij}^{k} \delta(m_{ij})^2+2y_{ij}^{k}\delta(m_{ij})+1\right)>0.
\end{equation}
In view of the relation above, the condition $\det(\mathbf{L}^k+\mathbf{Z}_{i j}(0))>0$ is equivalent to $-\Delta_{i j}^k\left(s_{i j}^k\right)^2-2 y_{i j}^k s_{i j}^k+1>0$. The latter condition is computationally easier to check.

\emph{Case (ii).} When $d_{ij}\neq 0$, 
the minimizer is given by
\begin{equation}\label{m_ij_case(ii)}
	m_{i j}=s_{i j}^k+\frac{y_{i j}^k}{\Delta_{i j}^k}+\frac{\Delta_{i j}^k-\sqrt{\left(\Delta_{i j}^k\right)^2+4 d_{i j}^2  y_{i i}^k y_{j j}^k}}{2 \Delta_{i j}^k d_{i j}}.
\end{equation}
Similarly we can also check a positive determinant.

Secondly, consider the case that $0$ is not in the domain of $\phi_{i j}(\cdot)$, i.e., $\det(\mathbf{L}^k+\mathbf{Z}_{i j}(0))\leqslant 0$. The minimizer of $\phi_{i j}(\cdot)$ is equal to $m_{ij}$ which is given by \eqref{m_ij_case(i)} if $d_{ij}=0$, and is given by \eqref{m_ij_case(ii)} otherwise.

%

We can now summarize the above results as follows.
\begin{itemize}
	\item When $\det(\mathbf{L}^k+\mathbf{Z}_{i j}(0))\leqslant 0$, define a mapping 
	\begin{equation}\label{A_map_det_negative}
		\Acal(s_{ij}^{k}):=m_{ij}.
	\end{equation}
	\item When $\det(\mathbf{L}^k+\mathbf{Z}_{i j}(0))> 0$, let
	\begin{equation}\label{sij1}
		{\Acal(s_{ij}^{k})}:=\left\{\begin{array}{cc}
			0 & \text { if } \phi_{i j}(0)<\phi_{i j}(m_{ij}) \\
			m_{i j} \cdot \mathbb{I} \left(s_{i j}^k \neq 0 \right)& \text { if } \phi_{i j}(0)= \phi_{i j}(m_{ij}) \\
			m_{i j} & \text { if } \phi_{i j}(0)>\phi_{i j}(m_{ij})
		\end{array}\right.
	\end{equation}
	\item Update $s_{ij}^{k+1}=\Acal(s_{ij}^k)$.
\end{itemize}


The full algorithm for the problem \eqref{L-S_update} is given next. 

\begin{algorithm}[H]
	\caption{Alternating minimization algorithm for \eqref{lagrange-type}}
	\label{main-alg}
	Input: $\lambda$, $\mu$, $\hat{\Sigmab}$, 
	 an upper bound for the number of iterations $\texttt{maxit}$, tolerance
	 level $\texttt{tol}$. \\
	 Set the iteration counter $k=0$, and initialize $\mathbf{L}^0$, $\mathbf{S}^0$. \\
	Output: the convergent iterate $(\mathbf{L}_{\opt}, \mathbf{S}_{\opt})$.
	\begin{algorithmic}
		\While{the stopping condition does not hold}
	\begin{enumerate}
			\item  Denote the current iterates are $\mathbf{S}^k=[s_{ij}^k]$, $\Sigmab^k=[d_{i j}]$ and $\mathbf{Y}^k=[y_{i j}]=(\mathbf{L}^k+\mathbf{S}^k)^{-1}$;
			\item Update $\mathbf{S}^{k+1}$. For each pair of $(i, j)$, $i,j=1, 2, \cdots, p$, do the following:  
			\begin{enumerate}
				\item[] (2.1) $i=j$. Compute $m_{ii}$ with \eqref{sii1}, and let
				\begin{equation}\label{A_map_diag}
					\mathcal{A}(s_{ii}^{k})=m_{ii}.
				\end{equation}
				\item[] (2.2) $i \neq j$. Compute $m_{i j}$ with \eqref{m_ij_case(i)} if $d_{ij}=0$ and with \eqref{m_ij_case(ii)} otherwise.
				\begin{itemize}
					\item[$\bullet$] If $-\Delta_{i j}^k\left(s_{i j}^k\right)^2-2 y_{i j}^k s_{i j}^k+1>0$, compute map $\Acal (s_{ij}^{k})$ with \eqref{sij1}.
					
					\item[$\bullet$] If $-\Delta_{i j}^k\left(s_{i j}^k\right)^2-2 y_{i j}^k s_{i j}^k+1 \leq 0$,
					compute map $\Acal (s_{ij}^{k})$ with \eqref{A_map_det_negative}.
				\end{itemize}
				\item[] (2.3) Update $s_{ij}^{k+1}$ (and $s_{ji}^{k+1}$ if $i\neq j$) with
				\begin{equation}
					s_{ij}^{k+1}=\mathcal{A} (s_{ij}^{k}).
				\end{equation}
			\end{enumerate}
			\item Update $\mathbf{L}^{k+1}$. Solve \eqref{L_update} for $\mathbf{L}^{k+1}$ using CVX. 
			\item Update $\mathbf{Y}^{k+1}=(\mathbf{L}^{k+1}+\mathbf{S}^{k+1})^{-1}$.
			\item  If $k<\texttt{maxit}$, increase $k$ by $1$.

		\end{enumerate}
		
			\EndWhile \\
			\Return the final iterate $\mathbf{L}^k$, $\mathbf{S}^k$.
	\end{algorithmic}
	
\end{algorithm}

\subsection{Complexity analysis}

The time complexity of Algorithm \ref{main-alg} is primarily determined by the updates of matrices $\Sb$, $\Lb$ and $\Yb$.  Since  $\Sb$ is symmetric, the complexity of updating $\Sb^{k+1}$ in each iteration is $\frac{1}{2}\mathcal{O}( p^2)$. In practice, solving for $\Lb^{k+1}$ using the SeDuMi or SDPT3 solvers in CVX has a complexity of approximately $\mathcal{O}( p^3)$. Additionally, the complexity of computing the inverse matrix  $\mathbf{Y}^{k+1}=(\mathbf{L}^{k+1}+\mathbf{S}^{k+1})^{-1}$ is at most $\mathcal{O}( p^3)$. Therefore, the total complexity of Algorithm \ref{main-alg} in each step is about  $\mathcal{O}( p^3)$.

\subsection{Initialization}\label{subsec:init}

Consider the spectral decomposition of $\hat{\mathbf{\Sigma}}=\Ub \Lambdab \Ub^\top$ where $\Lambdab = \diag(\lambda_{1}, \lambda_{2}, \cdots, \lambda_{p})$ such that $\lambda_{i} >0$, $i=1,2,\dots,p$ are the eigenvalues in decreasing order. 
Let 
$$\tilde\Lambdab = \diag(\lambda_{1}, \lambda_{2}, \cdots, \lambda_{t}, 0, \cdots, 0)$$
with $t<p$, and we initialize $\mathbf{L}^0$ as $\Ub \tilde\Lambdab \Ub^\top$. It is suggested in \cite{marjanovic2015l0} that the initialization of $\Sb$ needs to have sufficient sparsity, so we set $\Sb^0$ as $\mathtt{diag}(\mathtt{diag}(\hat{\mathbf{\Sigma}}-\Ub \tilde\Lambdab \Ub^\top))$, where the operator $\mathtt{diag}$ here refers to the Matlab command.

\subsection{Stopping criteria}\label{subsec:stop}

In our implementation, Algorithm \ref{main-alg} is terminated when the difference between two successive iterates is sufficiently small, i.e.,
\begin{equation}
	\|(\Lb^{k+1}, \Sb^{k+1}) - (\Lb^k, \Sb^k)\|_\F < \texttt{tol}
\end{equation}
where $\texttt{tol}>0$ represents the tolerance level.

\section{Simulation Results}\label{sec:sims}

	In this section we give simulation results of our algorithm on synthetic and real datasets.
	
\subsection{Synthetic data examples}
In this subsection, we evaluate the estimation performance of the algorithm through Monte Carlo simulations. For any fixed parameters $\lambda$  and $\mu$, we repeat the following four steps $100$ times:
\begin{enumerate}
	\item[(1)] According to the FA model~\eqref{generate_model}, randomly generate a diagonal matrix $\Sb\succ0$ and a factor loading matrix $\Gammab$ with $r$ columns. Then 
	generate $N$ samples where the cross-sectional dimension is $p = 40$; 
	\item[(2)] Compute the sample covariance matrix $\hat{\Sigmab}$ using the average in \eqref{sample_cov};
	\item [(3)] Divide the samples  into training set and testing set of equal size, and compute the corresponding score function value to select the parameters combination;
	\item[(4)] Use 
	the procedure described in the previous section to compute $\Lb_{\opt}$ and $\Sb_{\opt}$;
	\item[(5)] Evaluate the numerical rank $r_{\opt}$ of $\Lb_{\opt}$ by applying the scheme proposed in \cite{ciccone2018factor}:
	\begin{equation}\label{factor_estimate}
		r_{\opt}:=\underset{i\leq i_{\max}}{\max}\ \frac{t_i}{t_{i+1}},
	\end{equation}	
	where $t_i$, $i =1,2,\cdots,p$, denotes the $i$-th singular value (eigenvalue) of $\Lb_{\opt}$ in decreasing order, and $i_{\max}$ is defined to be the first $i$ satisfying $t_{i+1}/t_{i}<0.05$.
\end{enumerate}

\begin{figure}[t]
	\begin{minipage}[b]{.48\linewidth}
		\centering
		\centerline{\includegraphics[width=4.9cm]{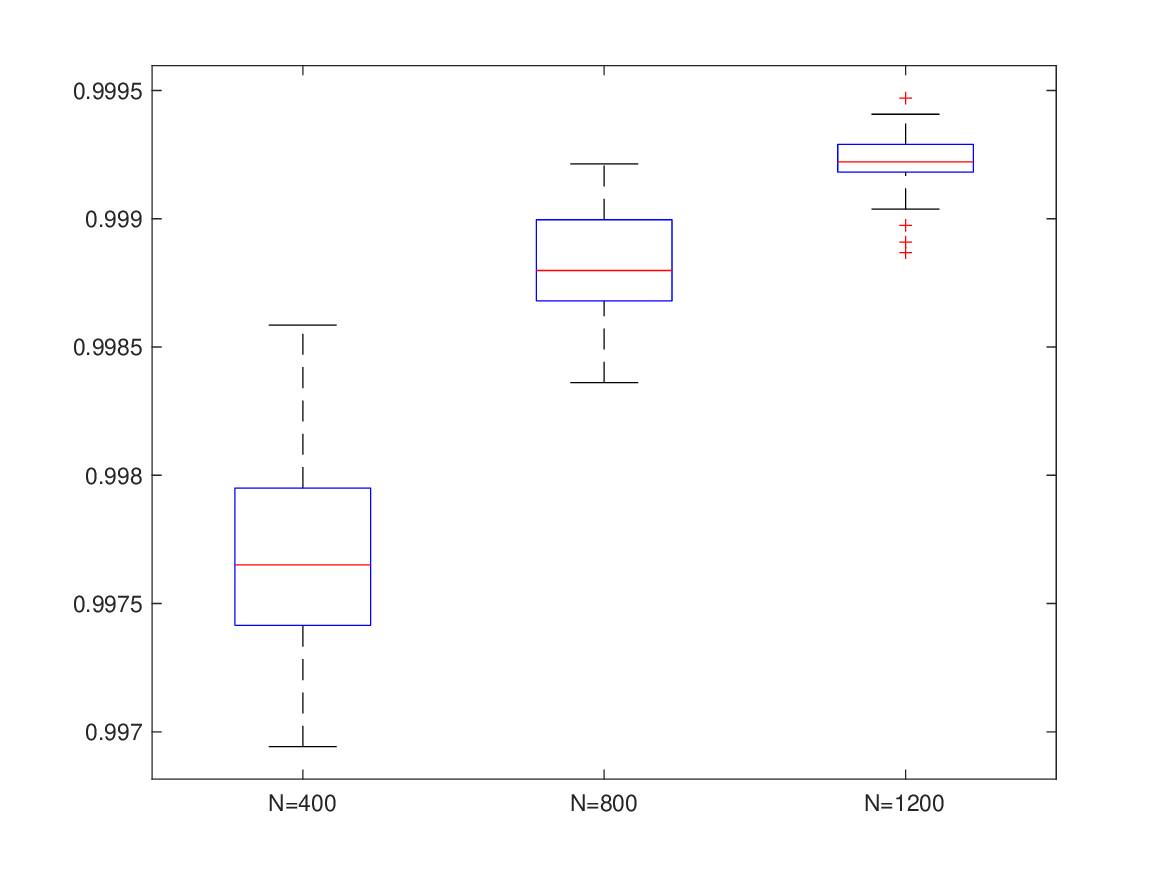}}
		\centerline{(a) $r=5$}
	\end{minipage}
	\hfill
	\begin{minipage}[b]{.48\linewidth}
		\centering
		\centerline{\includegraphics[width=4.9cm]{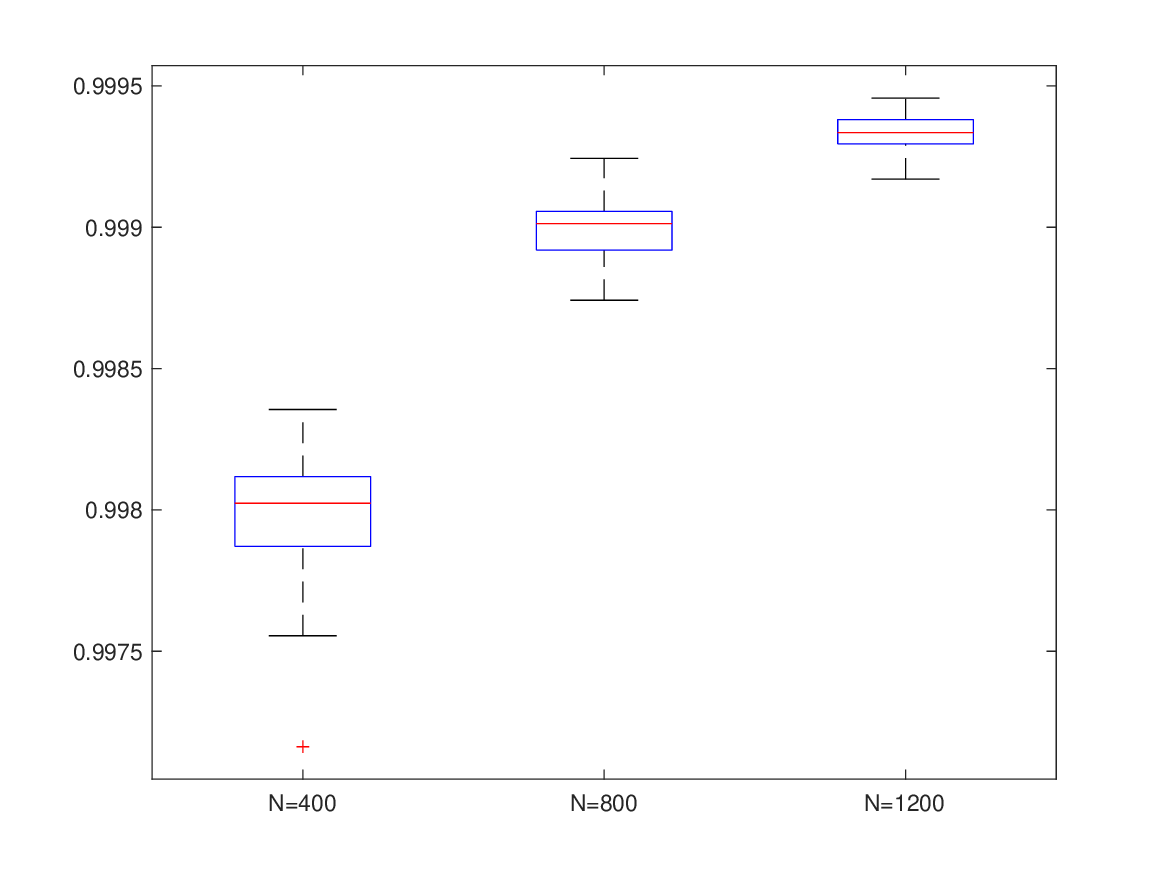}}
		\centerline{(b) $r=10$}
	\end{minipage}
	\caption{Recovery performance as measure by \eqref{ratio} under different sample sizes $N=400, 800, 1200$ for two cases with the true rank $r=5$ and $10$, respectively. 
		The regularization parameters $(\lambda, \mu)$ for each independent trial are selected via the CV procedure on a randomly generated dataset.
		}
	\label{fig:ratio}
\end{figure}
	\noindent\textbf{Parameters Setting.} We adopt the default values $(\texttt{tol}, \texttt{maxit})= (10^{-3}, 10^3)$ for all the algorithms mentioned in this section. We utilize Cross-Validation (CV) to choose the parameter values, with the tuning parameter $\lambda$ and penalty parameter $\mu$ sweeping over the arithmetic progression $\{10,35,60,\cdots,210\}$. We define the following score function and select the parameters combination corresponding to the minimum function value:
	\begin{equation}
		\textrm{score} := (r_{{\Lb}^{\trm}}+\|\Sb_{\trm}\|_{0}) \Dcal_{\KL} (\Lb_{\trm}+\Sb_{\trm} \| \hat{\Sigmab}^{\vrm} ),
	\end{equation}
	where $\hat{\Sigmab}^{\vrm}$ denotes the sample covariance matrix on the validation set, $(\Lb^{\trm}, \Sb^{\trm})$ represents the optimal solution on the training set, and $r_{\Lb^{\trm}}$ is the numerical rank of $\Lb^{\trm}$ computed by \eqref{factor_estimate}. 

\noindent\textbf{Evaluation~Protocol.}
We consider the recovery performance index for the subspace of $\Lb$ proposed in \cite{ciccone2018factor}: 
\begin{equation}\label{ratio}
	\text{ratio}(\mathbf{\Gamma}_{\opt}):= \frac{\operatorname{tr}(\mathbf{\Gamma}^{\top}\Pb\mathbf{\Gamma})}{\operatorname{tr}(\mathbf{\Gamma}^{\top}\mathbf{\Gamma})}
\end{equation}
which has a value between $0$ and $1$. The larger the ratio is, the better the subspaces are aligned. In the above formula, $\Gammab$ is the loading matrix generated in Step (1), $\Gammab_{\opt}$ is such that $\Lb_{\opt}=\mathbf{\Gamma}_{\opt}\mathbf{\Gamma}_{\opt}^{\top}$ which can be computed from the eigenvalue decomposition,
and $\Pb$ is the projection matrix onto the column space of $\Gammab_{\opt}$.
Fig.~\ref{fig:ratio} shows the boxplots of the values of \eqref{ratio} in $100$ repeated trials for the true rank $r=5$ (left panel) and $r=10$ (right panel), respectively. One can see that the recovery ratios are all above $0.99$ and get closer to $1$ as the sample size $N$ increases.

Using a particular instance of the dataset with $N=1200$ and $\Sb=\Ib$, we compare the estimation performance of the $\ell_0$ and $\ell_1$ regularizations on the sparse component $\Sb$. Fig.~\ref{compare_0/1} illustrates that the $\ell_0$ model effectively estimates the diagonal structure of $\Sb$, but the $\ell_1$ model provides a null matrix which is obviously wrong. Fig.~\ref{convergence} compares the \emph{empirical} convergence rate of Algorithm \ref{main-alg} with the ADMM \cite{FA_SAM_24} on the same dataset with $N=1200$, where the value of $(\lambda,  \mu)$ is $(10,160)$. Clearly, the results reveal that the convergence rate of the alternating minimization algorithm is linear while the ADMM is only sublinear.


\begin{figure}[t]
	\begin{minipage}[b]{.48\linewidth}
		\centering
		\centerline{\includegraphics[width=5.0cm]{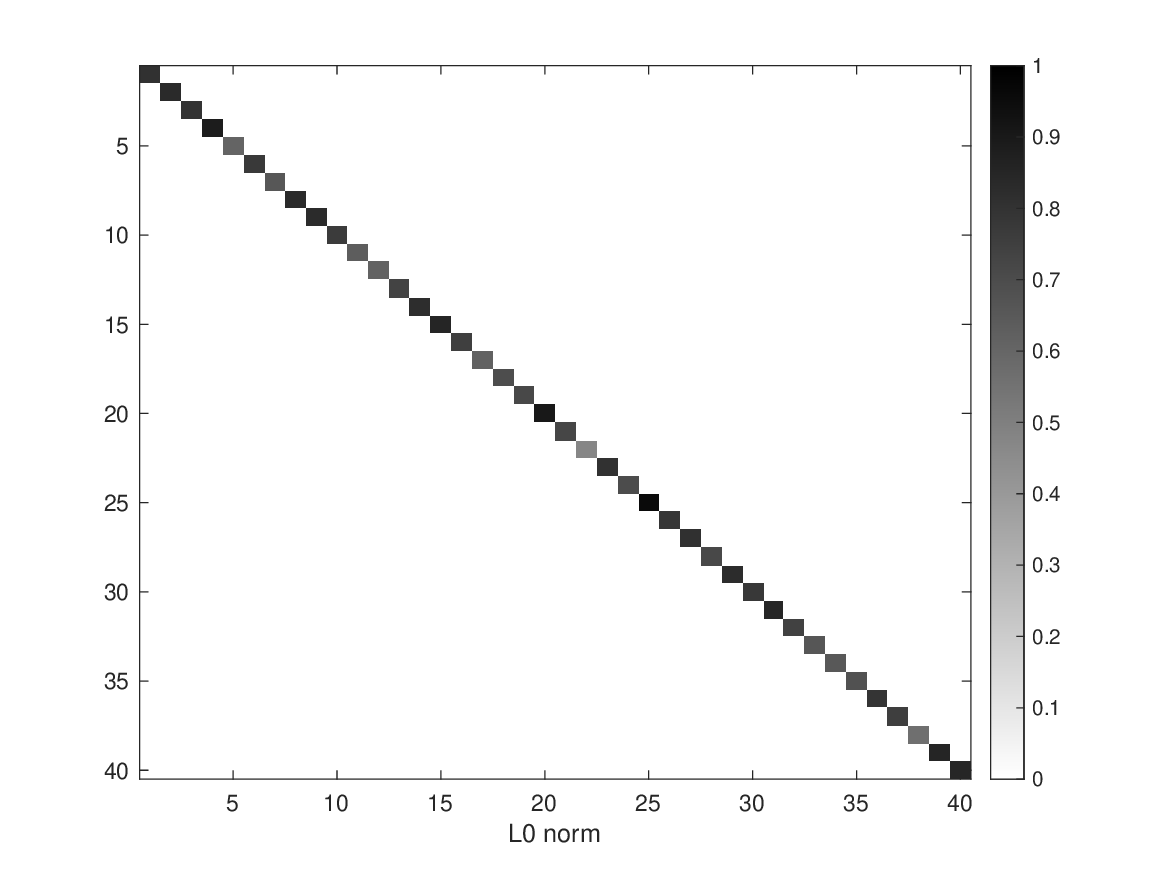}}
	\end{minipage}
	\hfill
	\begin{minipage}[b]{.48\linewidth}
		\centering
		\centerline{\includegraphics[width=5.0cm]{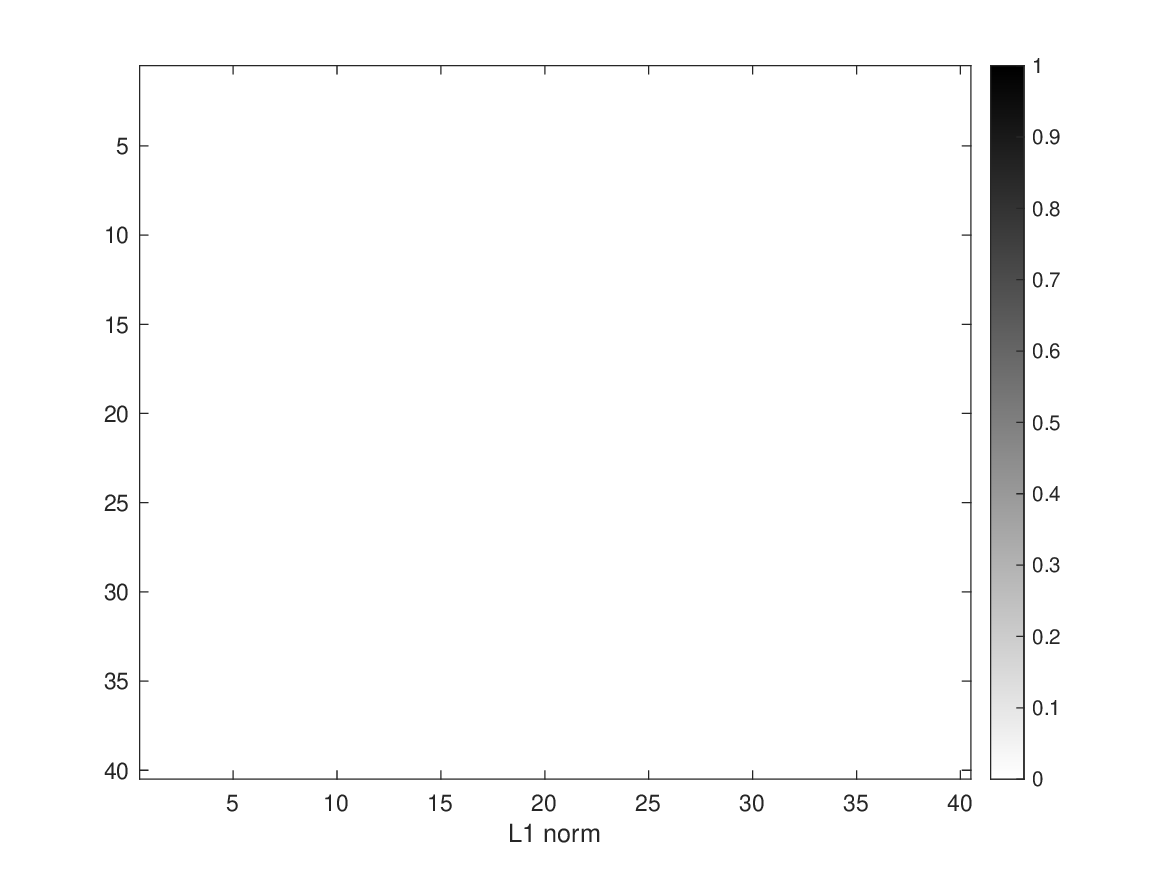}}
	\end{minipage}
		\caption{A comparison of the estimates of $\Sb$ under the $\ell_0$ and $\ell_1$ regularizations when the true $\Sb$ is the identity matrix. The regularization parameters $(\lambda, \mu)$ are $(10, 210)$ in both cases.}
	\label{compare_0/1}
\end{figure}

\begin{figure}[t]
	\centering
	\includegraphics[width=8.1cm]{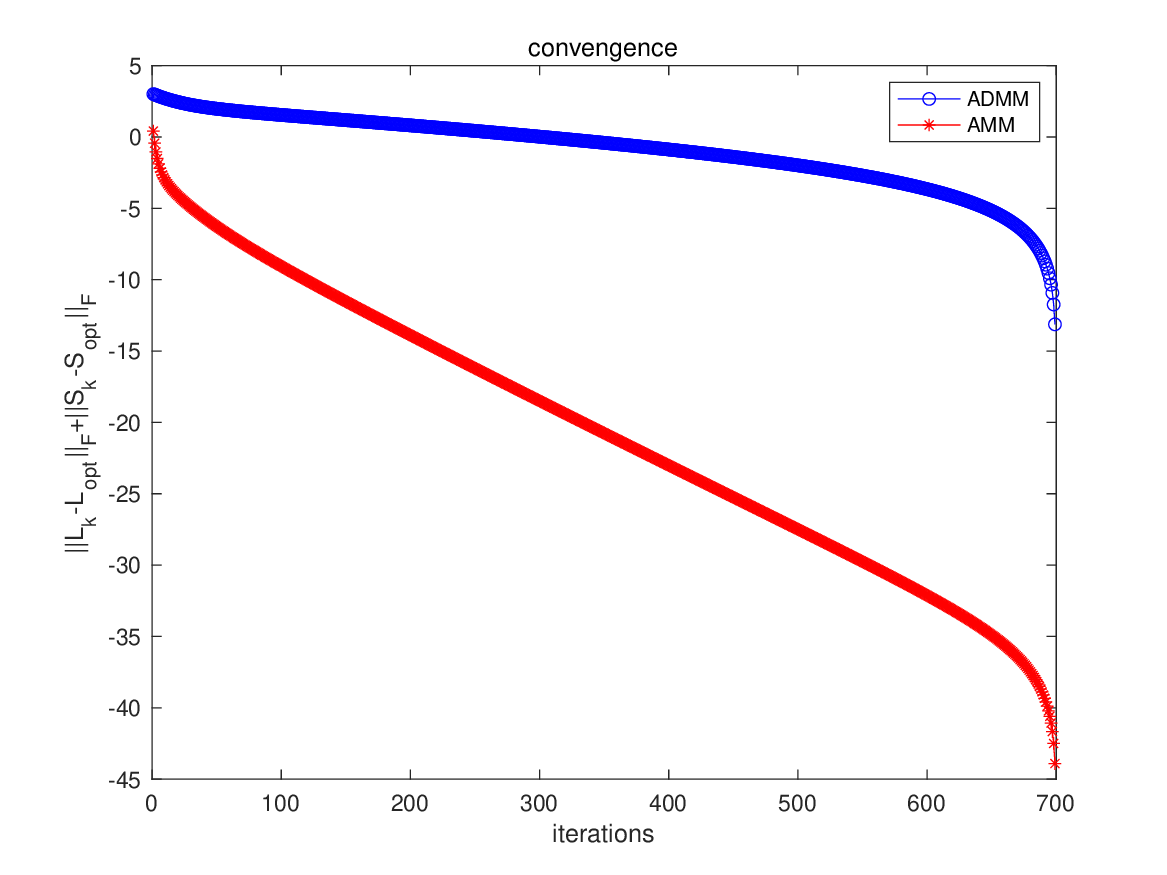}
	
	\caption{The convergence behavior of the proposed  algorithm and ADMM algorithm.
	}
	\label{convergence}
\end{figure}

\subsection{Real data examples}

In this subsection, we analyze a dataset consisting of nine financial indicators ($p=9$) collected from $92$ different sectors ($N = 92$) of the U.S. economy. Each data vector represents the average values for the respective sector\footnote{The dataset was sourced from  \url{http://www.stern.nyu.edu/~adamodar/pc/datasets/betas.xls} (downloaded in May 2023).}.  

There indicators include the beta which represents the systemic risk associated with general market movements, the Hi-Lo risk, the unlevered beta, the unlevered beta corrected for cash, the standard deviation of equity, the standard deviation of operating income, the debt/equity ratio,
the effective tax rate, and the cash/firm value ratio.
\begin{figure}[t]
	\centering
	\includegraphics[width=8.1cm]{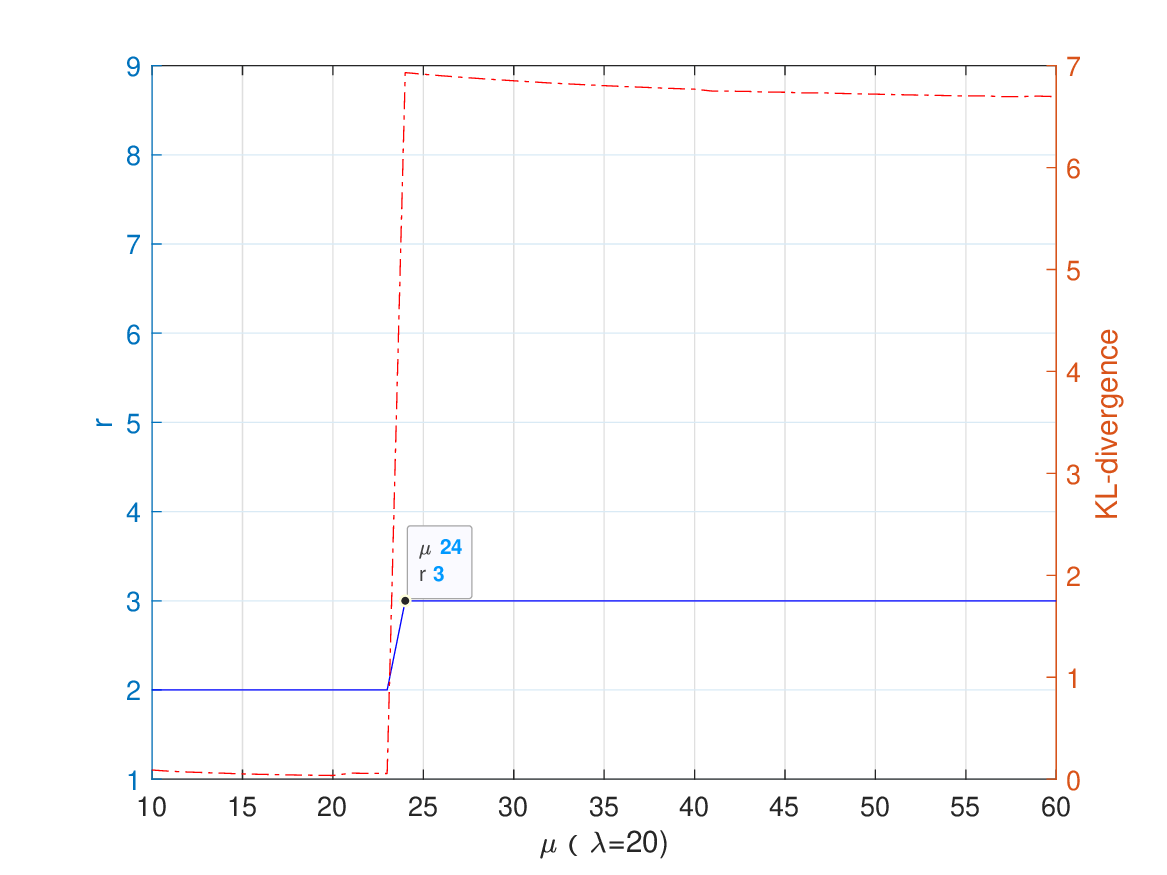}
	
	\caption{Values of $r_{\opt}$ and $\Dcal_{\KL}(\Lb_{\opt}+\Sb_{\opt} || \hat{\Sigmab})$ as a function of $\mu=10, 11, 12, \dots, 60$ under a fixed $\lambda=20$.
	}
	\label{real_rank_kl}
\end{figure}
As revealed from our simulations,
the value of $r_{\opt}$ mainly depends on the parameter $\mu$. Fig.~\ref{real_rank_kl} shows the change of $r_{\opt}$ and
$\Dcal_{\KL}(\Lb_{\opt}+\Sb_{\opt} || \hat{\Sigmab})$
with respect to $\mu$ under a fixed $\lambda=20$. The dashed line represents the KL divergence curve, while the solid line represents the rank curve. We conclude that our algorithm is quite robust in the estimation of the numerical rank, because the rank curve (blue line in Fig.~\ref{real_rank_kl}) is very flat with the only change happening at $\mu=23, 24$. A similar observation can be made from the red dashed curve.


\section{Conclusion}\label{sec:conc}

In this paper, we have considered the additive decomposition problem of low-rank and sparse matrices in Factor Analysis which is formulated as nonconvex nonsmooth optimization problem involving the $\ell_0$ norm and the KL divergence. 
We have proposed an alternating minimization scheme for the solution of the optimization problem.
Our algorithm can give a robust estimate of the number of common factors from the data, which has been verified in the numerical experiments.

\appendix


\begin{proof}[Proof of Theorem~\ref{thm_solution_compare}]
	Suppose that $(\Lb,\Sb)\in  \Qcal_{\ell_0} (\lambda)$. Then for any symmetric perturbation matrix $\Deltab$,
	there exists a positive constant $\epsilon_s$ such that
	\begin{equation}\label{local_min_s}
	H(\Lb, \Sb+\Deltab)\geq H(\Lb, \Sb) \text { with}\ \|\Deltab\|_\F < \epsilon_s,  
	\end{equation}
	where $\Deltab_{p \times p}=\left[\delta_{i j}\right]$.
	Let $(\Lb+\Sb)^{-1}=\left[y_{i j}\right]$ and ${\hat{\Sigmab}}^{-1} = \left[d_{i j}\right]$.
	Building upon the above inequality, next we derive necessary optimality conditions for the problem $ \eqref{Object}$. 
	
	%
	
	Let $\Zcal(\Sb) = \{ (i,j): s_{ij}  \neq 0 \}$ represent the set of indices corresponding to the non-zero elements in $\Sb$. For any $(i,j)\in \Zcal(\Sb)$, we consider 
	\begin{equation}
	\Deltab = 
	\left\{\begin{array}{cc}
	\delta_{ii} \eb_i \eb_i^{\top}& \text { if } i=j \\
	\delta_{ij} \Ub_{ij}\Ub_{ji}^{\top}& \text { if } i\neq j. 
	\end{array}\right.
	\end{equation}
	\begin{itemize}
		\item  If $i=j$, then after straightforward computation we arrive at
		\begin{equation}
		\begin{aligned}
		&H(\Lb,\Sb+\Deltab)-H(\Lb,\Sb) \\
		&=\mu\left[ -\log(1+\delta_{i i} y_{i i}) +\delta_{i i} d_{i i}   \right].
		\end{aligned}
		\end{equation}
		\item  If $i \neq j$, then similarly we have
		\begin{equation}
		\begin{aligned}
		&H(\Lb,\Sb+\Deltab)-H(\Lb,\Sb) \\
		&=\mu\left[ -\log( -c_{i j}\delta_{i j}^2 +2y_{i j} \delta_{i j} +1 ) +2 \delta_{i j} d_{i j} \right]  \\
		&\quad \quad\quad\quad\quad\quad\quad+2 \lambda \mathbb{I}(s_{i j}+\delta_{i j} \neq 0 )-2\lambda \mathbb{I}(s_{i j} \neq 0 ) ,
		\end{aligned}
		\end{equation}
		where $c_{i j}=y_{i i} y_{i j} - y_{i j}^2$.
	\end{itemize}
	
	Now we restrain ourselves to a smaller perturbation inside the ball $\|\Deltab\|_\F<\epsilon_s$ such that $|\delta_{i j}|<\min \{| s_{i j } |,\epsilon_s/2 \}$ for all $(i, j)\in\Zcal(\Sb)$
	and $\delta_{i j}=0$ for $s_{ij}=0$. Then we have $\mathbb{I}(s_{i j}+\delta_{i j} \neq 0 ) = 1$
	when $i\neq j$ and $(i,j)\in\Zcal(\Sb)$. 
	Let us define a function $h\left(\delta_{i j}\right) :=$
	\begin{equation}
	\begin{cases}\mu \left[  -\log \left(1+\delta_{i i} y_{i i}\right)+ \delta_{i i} d_{i i}  \right] & \text { if } i=j \\
	\mu \left[  -\log \left(-c_{i j} \delta_{i j}^2+2 y_{i j} \delta_{i j}+1\right)+2 \delta_{i j} d_{i j} \right] & \text { if } i \neq j\end{cases}
	\end{equation}
	It follows from the condition \eqref{local_min_s} that $h(\delta_{i j})\geq 0$ in a small neighborhood of $\delta_{i j}=0$ such that $h(0)=0$. Therefore, according to Fermat's lemma, we have $h'(0)=0$. The derivative of $h$ is just 
	\begin{equation}
	h'\left(\delta_{i j}\right)=\left\{
	\begin{aligned}
	&d_{i j}+(-y_{i i})/(1+\delta_{i i} y_{i i})  && \text { if } i=j \\
	&2 d_{i j}+\frac{2 c_{i j} \delta_{i j}-2 y_{i j}}{-c_{i j} \delta_{i j}^2+2 y_{i j} \delta_{i j}+1} && \text { if } i\neq j,
	\end{aligned}\right.
	\end{equation}	
	and the stationary-point condition reads as 
	\begin{equation}\label{optimal_cond_l0}
	-y_{i j}+d_{i j}=0,\ \text{for any}\ (i,j)\in \Zcal(\Sb).
	\end{equation}

	Subsequently, we suppose that $(\hat{\Lb},\hat{\Sb})\in \Qcal_{\ell_1}(\tau)$. By a global minimizer, we have
	\begin{equation}\label{l1_subpro_S}
	\hat{\Sb} = \arg\min_{\Sb\succ 0}\ \mu \left[\trace (\hat{\Lb}+\Sb) {\hat{\Sigmab}}^{-1} -\log\det(\hat{\Lb}+\Sb)\right]+\tau \|\Sb\|_1.
	\end{equation}
	Since objective function of \eqref{l1_subpro_S} is convex and continuous with respect to $\Sb$, $(\hat{\Lb},\hat{\Sb})$ satisfies the following stationary-point equation:
	\begin{equation}
	\mu \hat{\Sigmab}^{-1} -\mu (\Lb+\Sb)^{-1}+\tau \Gammab=0,\quad \Gammab_{i j}=\sign{s_{i j}},
	\end{equation}
	which is
	\begin{equation}\label{optimal_cond_l1}
	\begin{cases}\mu d_{i i}-\mu y_{i i}+\tau=0, & i=j \\ \mu d_{i j}-\mu y_{i j}+\tau \sign{s_{i j}}=0, & i \neq j .\end{cases}
	\end{equation}
	If $(\hat\Lb, \hat{\Sb})  \in  \Qcal_{\ell_0} (\lambda)$ were true, it would be required that both \eqref{optimal_cond_l0} and \eqref{optimal_cond_l1} hold simultaneously for some $\tau>0$. However, this is clearly not possible, and the proof is completed.
\end{proof}

\bibliographystyle{IEEEtran}
\bibliography{refs}

\end{document}